\documentclass[twoside]{article}

\usepackage{amsmath,amsfonts,amssymb,amscd,amsthm,amsbsy}
\usepackage{fancyhdr,latexsym,epsfig,enumerate}

\usepackage[latin1]{inputenc}
\usepackage[T1]{fontenc}
\usepackage[matrix,arrow]{xy}
\usepackage{endnotes}

\makeatletter

\def\ps@myheadings{%
     \let\@oddfoot\@empty\let\@evenfoot\@empty
     \def\@evenhead{\thepage\hfil\MakeUppercase{\footnotesize\leftmark}\hfil}%
     \def\@oddhead{{\hfil\MakeUppercase{\footnotesize\rightmark}}\hfil\thepage}%
     \let\@mkboth\@gobbletwo
     \let\sectionmark\@gobble
     \let\subsectionmark\@gobble
     }
\makeatother

\newtheorem{thm}{Theorem}[section]
\newtheorem{lem}[thm]{Lemma}
\newtheorem{prop}[thm]{Proposition}
\newtheorem{cor}[thm]{Corollary}

\newtheorem{defin}[thm]{Definition}
\theoremstyle{definition}
\newtheorem{rem}[thm]{Remark}

\newtheorem{exo}[thm]{Example}

\newcommand{\surf}[2]{\widehat{\Sigma}_{{#1}{#2}}}
\newcommand{\bsurf}[3]{\mathop{B_{#1}(\widehat{\Sigma}_{{#2}{#3}})}}

\newcommand{\relem}[1]{Lemma~\protect\ref{lem:#1}}
\newcommand{\repr}[1]{Proposition~\protect\ref{prop:#1}}

\newcommand{\brak}[1]{\ensuremath{\left\{ #1 \right\}}}

\newcommand{\sset}[2]{\ensuremath{\brak{#1 \, , \ldots, \, #2}}}

\newcommand{\N}{\ensuremath{\mathbb N}}
\newcommand{\Z}{\ensuremath{\mathbb Z}}
\newcommand{\D}{\ensuremath{\mathbb D}}

\newcommand{\FF}{\ensuremath{\mathbb F}}
\newcommand{\F}[1][n]{\ensuremath{\F\FF_{{#1}}}}

\renewcommand{\epsilon}{\ensuremath{\varepsilon}}
\renewcommand{\phi}{\ensuremath{\varphi}}
\renewcommand{\to}{\ensuremath{\longrightarrow}}

\newcommand{\MB}[3][,n]{\ensuremath{B_{{#2}{#1}}({#3})}}

\def\tia{\tilde{a}}
\def\tib{\tilde{b}}
\def\z{\zeta}
\def\si{\sigma}
\def\tsi{\tilde{\sigma}}
\def\qk{P_k}
\def\bqk{\bar{P}_k}

\begin{document}

\pagestyle{myheadings}

\markboth{Paolo Bellingeri, Eddy Godelle and John Guaschi}{Exact sequences, lower central series and representations of surface braids    \hspace{2pt}  }

\title{Exact sequences, lower central series and representations of surface braid groups}

\author{Paolo Bellingeri, Eddy Godelle and John Guaschi}

\date{}

\maketitle

\begin{abstract}
We consider exact sequences and lower central series of (surface) braid groups and we explain how they can prove to be useful for obtaining representations for surface braid groups. In particular, using a completely algebraic framework, we describe the notion of extension of a representation introduced and studied recently by An and Ko and independently by Blanchet.
\end{abstract}
\begingroup
\renewcommand{\thefootnote}{}
\endgroup 

\maketitle


\section{Introduction}


The faithful linear representations of Bigelow-Krammer-Lawrence of the Artin braid groups $B_n$ are probably one of the most important recent discoveries in the theory of braid groups, and as such, have been intensively studied over the last few years. They have also been extended to other groups, such as Artin-Tits groups of spherical type (see for instance~\cite{CW,D}). However, with the exception of a few results~\cite{Bar2,BiB,GG7}, their generalisation in a more topological direction, to braid and mapping class groups of surfaces for example, as well as the linearity of these groups, are open problems in general. 

\medskip

The aim of this paper is twofold. The first is to underline the relevance of short exact sequences of braid groups and their generalisations to the study of representations of these groups. As we shall recall in Section~\ref{section2}, the Burau and Bigelow-Krammer-Lawrence representations appear when one studies certain `mixed' extensions of $B_{n}$ arising from fibrations at the level of configuration spaces. This extension splits, and in the simplest case, its kernel is the fundamental group of the $n$-punctured disc $\mathbb{D}_{n}$. The Burau representation then occurs as the induced action of $B_{n}$ on the homology of the infinite cyclic covering $\widetilde{\mathbb{D}}_{n}$ of $\mathbb{D}_{n}$. 

\medskip  

The Bigelow-Krammer-Lawrence representations may be obtained in a similar way in terms of the Borel-Moore middle homology group of a $\Z^{2}$-covering of $\mathbb{D}_{n}$ (see Section~\ref{section2.3} for more details). Motivated by these constructions, one possible strategy for obtaining linear representations of (surface) braid groups is to study exact sequences of these groups. With this in mind, we recall the definitions of surface braid groups and their short exact sequences in Section~\ref{section3}, as well as the known results concerning the splitting of these sequences (Lemma \ref{lem:fibrations}). 
In the case of the short exact sequence of `mixed' Artin braid groups, the induced short exact sequence on the level of commutator subgroups brings into play groups and homomorphisms that appear in the construction of Bigelow-Krammer-Lawrence representations of $B_{n}$ (see the end of Section~\ref{section2.3}).

Before stating similar results for exact sequences of surface braid groups, in Section~\ref{section4}, we recall the basic definitions pertaining to the lower central series $(\Gamma_{i}(G))_{i\in \N}$ of a group $G$, and we show that if $G$ is the braid or mixed braid group (with a sufficiently large number of strings) of a compact orientable surface $\surf{g}{}$ with a single boundary component, the quotient group $G/\Gamma_{3}(G)$ is a semi-direct product of free Abelian groups. We provide group presentations for such quotients (Lemma \ref{lem:presgamma} and Corollary \ref{cor:decgammamix}) that will play an important r\^ole in the rest of the paper.

In~\cite{HK}, An and Ko described an extension of the Bigelow-Krammer-Lawrence representations of $B_{n}$ to braid groups of orientable surfaces of positive genus and with non-empty boundary. However, it is not currently known whether these representations are faithful.  The representation is based on the regular covering arising from
a projection of the $n$-th braid group of a surface $\Sigma$ with non-empty boundary onto a specific group  $G_\Sigma$, constructed in a technical manner in order to satisfy certain homological constraints (Section~3  and Definition 2.2 of  \cite{HK}) and which  turns out to be a Heisenberg group;  more recently, the above projection  has been independently studied by Christian Blanchet to obtain a representation of a large subgroup of the Torelli group of a surface with one boundary component  containing the Johnson subgroup \cite{Bl}.

The second aim of our paper is to show that  technical construction proposed in~\cite{HK} of the representations may be described in terms of lower central series and exact sequences of surface braid groups: this is the object of  Section~\ref{section5}.  As we mentioned above, in the case of Artin braid groups, the induced short exact sequence on the $\Gamma_{2}$-level gives rise to elements used in the construction of the Bigelow-Krammer-Lawrence representations. In the case of surface braid groups, the corresponding construction on the same level does not work (Proposition \ref{nolinearext}), but  if we take this construction a stage further, to the $\Gamma_{3}$-level, we obtain the corresponding objects of the An-Ko representations. More precisely   we show how to use the $\Gamma_{3}$-level to extend Bigelow-Krammer-Lawrence representations and we prove that such extensions are unique up to isomorphism (Propositions  \ref{fondpropex} and \ref{fondprop}), according to Definition \ref{fonddef}.

We finish the main part of the paper with another possible application of the lower central series of surface braid groups  by showing that  the standard length function on $B_{n}$ admits a unique extension to a homomorphism whose source is the braid group of a surface of positive genus with one boundary component (Proposition \ref{fondpropcons}), and that there is no such extension if the surface is closed and orientable (Proposition \ref{fondpropcons2}). In an Appendix, we discuss the relationship between the splitting of the `mixed' surface braid group sequences and that of their restriction to the corresponding pure braid groups. In the terminology of~\cite{FR}, we show that this restriction never give rises to an \emph{almost-direct} product (an almost-direct product structure means that the extension is split, and that the induced action of the lower central series quotient $K/\Gamma_{2}(K)$ on the kernel $K$ is trivial). 

\subsection*{Acknowledgements}

The authors were partially supported by the~ANR project TheoGar `Th\'eorie de Garside' n\textsuperscript{o} ANR-08-BLAN-0269-02.


\section{Exact sequences, lower central series and representations of classical braid groups} \label{section2}



\subsection{Surface braid groups and configuration spaces}


Surface braid groups are a natural generalisation of  both the classical
braid groups and of the fundamental group of surfaces.
First defined by Zariski during the 1930's, they were re-discovered
by Fox during the 1960's, and have been used subsequently in the study of
mapping class groups.

We recall the definition due to Fox of these groups in terms of
fundamental groups of configuration spaces~\cite{FoN}. Let $\Sigma$ be a connected surface. Let $\FF_n(\Sigma)=\Sigma^n
\setminus \Delta$, where $\Delta$ is  the set of $n$-tuples $x=(x_1,
\dots, x_n)\in \Sigma^n$ for which $x_i=x_j$ for some $i \not= j$. The
fundamental group $\pi_1(\FF_n(\Sigma))$ is called the \emph{pure
braid group} on $n$ strands of the surface $\Sigma$; it shall be
denoted by $P_n(\Sigma)$. There is a natural action of the symmetric
group $S_n$ on $\FF_n(\Sigma)$ by permutation of coordinates; the
fundamental group $\pi_1(\FF_n(\Sigma)/S_n)$ is called the \emph{braid
group} on $n$ strands of the surface $\Sigma$ and shall be denoted by
$B_n(\Sigma)$. Then $\FF_n(\Sigma)$ is a regular $n!$-fold covering of
$\FF_n(\Sigma)/S_n$ that gives rise to the following short exact sequence:

\begin{equation}\label{eq:permutation}
1\to P_n(\Sigma) \to B_n(\Sigma) \to S_{n}\to 1.
\end{equation}

Regarded as a subgroup of $S_{k+n}$, the group $S_k
\times S_n$ acts on $\FF_{k+n}(\Sigma)$. The fundamental group
$\pi_1\left(\FF_{k+n}(\Sigma)/(S_k \times S_n)\right)$ will be called
the \emph{mixed braid group of $\Sigma$ on $(k,n)$ strands}, and shall
be denoted by $\MB{k}{\Sigma}$. Notice that $\MB{k}{\Sigma}$ embeds canonically in $B_{k+n}(\Sigma)$.
These intermediate groups between pure braid and braid groups
of a surface, known as `mixed' braid groups, play an important r\^ole in~\cite{HK}. They were defined
previously in~\cite{GG2,Man,PR}, and were studied in more detail
in~\cite{GG4} in the case where $\Sigma$ is the $2$-sphere
$\mathbb{S}^2$.


\subsection{Fibrations and induced exact sequences}


We recall that  $\pi_1(\FF_n(\D^2))$ is isomorphic to the pure
braid group on $n$ strands, usually 
denoted by $P_n$, while $\pi_1(\FF_n(\D^2)/S_n)$ is isomorphic to  the braid
group  $B_n$ on $n$ strands. In what follows, 
the mixed braid group on $(k,n)$ strands
$\pi_1\left(\FF_{k+n}(\D^2)/(S_k \times S_n)\right)$  will be denoted simply by $B_{k,n}$.

\noindent The fibration $\FF_{k+n}(\D^2)/(S_k \times S_n) \to \FF_n(\D^2)/S_n$
given by forgetting the first $k$ coordinates is a locally-trivial
fibration whose fibre over a point $\{ x_1, \ldots, x_n \}$ may be identified with the orbit space
$\FF_k(\D^2 \setminus \sset{x_{1}}{x_{n}} )/S_{k}$.
From now on we denote $\D^2 \setminus \sset{x_{1}}{x_{n}}$ by $\D_n$.
Let us also denote  $\pi_1(\FF_k(\D_n )/S_{k})$
by $B_k(\D_n) $; this group turns out to be isomorphic to the subgroup of $B_{k+n}$ consisting of braids where the last $n$ strands are trivial (vertical).
The long exact sequence in homotopy of the above fibration yields a short exact sequence.

\begin{lem} 
Let $k,n\in \mathbb{N}$. The Fadell-Neuwirth fibration
$\FF_{k+n}(\D^2) \to \FF_{n}(\D^2)$
induces the short exact sequence:
\begin{equation}\label{eq:sequenceclassic}
1 \to B_k(\D_n) \to B_{k,n}
\to B_n  \to 1 \, . \tag{MB}
\end{equation}
\end{lem}

In a similar way, we may obtain the more well-known short exact sequence of  pure braid groups.
\begin{equation}
1 \to P_k(\D_n)  \to P_{k+n}
\to P_n \to 1.\tag{PB}
\end{equation}
Here $ P_k(\D_n)$ denotes the fundamental group of $\FF_k(\D_n)$, which is isomorphic to the subgroup of $P_{k+n}$ consisting of pure braids where the last $n$ strands are vertical.  Notice that the short exact sequences $(MB)$ and $(PB)$ split for all $k \ge1$, where the section is given geometrically by adding $k$ trivial strands `at infinity' (see for instance \cite{HK,B}).


\subsection{Linear representations for the braid group $B_n$} \label{section2.3}


When $k=1$, the short exact sequence $(PB)$ plays a central r\^ole in the study of Vassiliev invariants of braid groups and of Lie Algebras related to pure braid groups. We refer to \cite{Pap} for the classical case and to \cite{BF,EV,GP} for  analogous results in the case of surface braid groups.

In what follows, we will focus on the relevance of such short exact sequences to linear representations of braid groups and their topological generalisations. Let us start with  the case $k=1$. The group $B_n$  may be interpreted as the mapping class group  of $\D_n$~\cite{Bir}. We thus obtain an action of $B_n$ on $\D_n$ that induces an action on $\pi_1(\D_n)$, the latter being isomorphic to the free group $F_n$ on $n$ generators. This action, which is faithful, coincides with the action by conjugation of $B_n$ on $B_1(\D_n)$ defined by the natural section of $(MB)$. In this way, we recover the famous Artin representation of the braid group $B_n$ as a subgroup of the group of automorphisms of $F_n$.  Analogously, we have an action of $P_n$ as the pure mapping class group of $\D_n$  on $P_1(\D_n) \simeq F_n$ what is faithful and coincides with the action by conjugation of $P_n$ on $P_1(\D_n)$ defined by the natural section of $(PB)$.
Composing the Artin representation with the Magnus representation associated to the length function $p_1: B_1(\D_n) \to \Z$ (see for instance \cite{Bar})
we obtain the (non reduced) Burau representation of $B_n$.
In the case of the pure braid group, we obtain the Gassner representation of $P_n$ (\cite{Bar}) in a similar way.

The Burau representation also has a homological interpretation (see for instance Chapter 3 of \cite{KT}). Furthermore, it admits certain generalisations. Indeed, for any $k \ge 1$ we may observe that $B_n$, regarded as the  mapping class group  of $\D_n$,  acts on $\FF_k(\D_n)/S_{k}$ and therefore on its fundamental group, $B_k(\D_n)$.
The induced action of $B_n$ on  $B_k(\D_n)$ coincides with the action by conjugation of $B_n$ on $B_k(\D_n)$
 defined by the natural section of $(MB)$. In order to look for (linear) representations, we consider regular coverings associated with normal subgroups of $B_k(\D_n)$, and we try to see if the induced action on homology is well defined. In other words, we wish to study surjections of $B_k(\D_n)$ onto a group $G_k$ subject to certain constraints. 

When $k=1$, we consider as before the length function $p_1: B_1(\D_n) \to G_1=\Z=\langle  t \rangle$. Since the action of $B_n$ on 
$B_1(\D_n)$ commutes with $p_1:  B_1(\D_n) \to \Z$, $B_n$ acts on the regular covering 
$\widetilde{\D}_n$ of $\D_n$. The induced action on the first homology group of $\widetilde{\D}_n$ is the (reduced)  Burau representation of $B_n$. 
 
For $k>1$,  let  $G_k$ be  the group $\Z^2=\langle q, t\rangle$. The corresponding morphism $p_k: B_k(\D_n) \to G_k$ for $k>1$  sends the classical braid generators $\sigma_1, \ldots, \sigma_k$ to $q$ and the generators $\zeta_1, \ldots, \zeta_n$, corresponding to the generators of $\pi_1(\D_n)$, to $t$. Since the action of $B_n$ on 
$B_k(\D_n)$ commutes with $p_k:  B_k(\D_n) \to G_k$, it turns out that $B_n$ acts on the regular covering of
 $\FF_k(\D_n)/S_k$, and the induced action on the Borel-Moore middle homology group of such a covering space is in fact the $k$th Bigelow-Krammer-Lawrence representation of $B_n$. In this way,  for $k>1$ we obtain faithful linear representations of $B_n$ (see \cite{Big} for $k=2$ and \cite{Z} for $k>2$). We refer the reader to \cite{KT} for a complete description of these constructions. In what follows, we first motivate the choice of the above projections $p_k: B_k(\D_n) \to G_k$ using the lower central series of the corresponding groups. We then explain how the study of the lower central series of surface braid groups can be used to obtain the representations given  in \cite{HK}. We will define also the notion of extension of a representation in a completely algebraic manner, and its obstructions.

If we wish to consider surjections of $B_k(\D_n)$ onto some group $G'_k$,
in order to obtain a linear representation using the approach described above, the group $G'_k$ should be Abelian. 
Considering the short exact sequence $(MB)$ on the level of Abelianisation, we obtain 
the following commutative diagram of short exact sequences:
\begin{equation}
\begin{xy}*!C\xybox{%
\xymatrix{
1 \ar[r]  &  B_k(\D_n)  \ar[d]^{\bar q_k}   \ar[r]    & B_{k,n}  \ar[d]^{r_{k,n}} \ar[r]          &  B_n   \ar[d]^{r_{n}} \ar[r]        & 1\\
1 \ar[r]  &  \ker \bar\psi_k  \ar[r]                          & B_{k,n}/ \Gamma_2(B_{k,n})  \ar[r] &  B_n/ \Gamma_2(B_n)  \ar[r]   & 1}}
\end{xy}
\end{equation}
where $r_{k,n}$ and $r_n$ are Abelianisation homomorphisms, and $\bar\psi_k$ is the homomorphism satisfying 
$\bar\psi_k \circ r_{k,n}=r_n \circ \psi_k$. It is straightforward to show that for all $k\ge 1$, $ \ker \bar\psi_k$ and $\bar{q}_k$ coincide respectively with the group $G_k$ and the morphism $p_k$ considered in the Bigelow-Krammer-Lawrence representations.


 \section{Exact sequences for surface braid groups} \label{section3}


We now return to the general case. Let $\Sigma$ be an orientable surface.
The map $\FF_{k+n}(\Sigma)/(S_k \times S_n) \to \FF_n(\Sigma)/S_n$
given by forgetting the first $k$ coordinates is a locally-trivial
fibration whose fibre may be identified with
$\FF_k(\Sigma \setminus \sset{x_{1}}{x_{n}} )/S_{k}$.
As in the case of pure braid groups~\cite{FaN}, the long exact
sequence in homotopy of this fibration yields a short exact sequence.

\begin{lem} 
Let $k,n\in \mathbb{N}$. The Fadell-Neuwirth fibration
$\FF_{k+n}(\Sigma) \to \FF_{n}(\Sigma)$
induces the short exact sequence:
\begin{equation}\label{eq:sequence}
1 \to B_k(\Sigma \setminus \sset{x_{1}}{x_{n}}) \to \MB{k}{\Sigma}
\to B_n(\Sigma) \to 1 \,,\tag{MSB}
\end{equation}
where we suppose that $n\geq 3$ if $\Sigma=\mathbb{S}^2$.
\end{lem}

In what follows we shall refer to the above short exact sequence
as~(\ref{eq:sequence}) (mixed surface braid groups sequence), and we denote its
restriction to the corresponding pure braid groups
\begin{equation}\label{eq:psequence}
1 \to P_k(\Sigma \setminus \sset{x_{1}}{x_{n}}) \to P_{k+n}(\Sigma)
\to P_n(\Sigma) \to 1 \,,\tag{SPB}
\end{equation}
by~(\ref{eq:psequence}) (surface pure braid groups sequence). If $\Sigma$ is
the disc $\D^2$, we recover the sequence considered in the previous
section that gives rise to the Bigelow-Krammer Lawrence representation. In a similar manner, we may
study~(\ref{eq:sequence}) in order to find representations of
$B_n(\Sigma)$.

We have the following commutative diagram involving the short exact
sequences~(\ref{eq:permutation}),~(\ref{eq:psequence})
and~(\ref{eq:sequence}):
\begin{equation} \label{eq:commsequence1}
\begin{xy}*!C\xybox{%
\xymatrix{%
& 1\ar[d] & 1\ar[d] & 1\ar[d] & \\
1 \ar[r] & P_{k}(\Sigma \setminus \sset{x_{1}}{x_{n}}) \ar[r]\ar[d] &
P_{k+n}(\Sigma) \ar[r]\ar[d]  & P_{n}(\Sigma) \ar[d] \ar[r] & 1 \\
1 \ar[r] & B_k(\Sigma \setminus \sset{x_{1}}{x_{n}}) \ar[r]\ar[d] &
\MB{k}{\Sigma} \ar[r] \ar[d] &  B_n(\Sigma) \ar[r]\ar[d] & 1\\
1 \ar[r] & S_{k} \ar[d]\ar[r] & S_{k}\times S_{n} \ar[d]\ar[r]  &
S_{n} \ar[d]\ar[r] & 1\\
& 1 & 1 & 1 &
}}
\end{xy}
\end{equation}
where the vertical arrows between~(\ref{eq:psequence})
and~(\ref{eq:sequence}) are inclusions, and the second vertical
sequence is obtained by restricting the exact sequence~(\ref{eq:permutation}) for $k+n$
strings to $\MB{k}{\Sigma} $. The third row of symmetric groups splits
as a direct product.

The following lemma summarises some of the known
results for the splitting problem for the sequences~(\ref{eq:psequence})
and~(\ref{eq:sequence}) (we refer to \cite{FaV,GG4} for the case of  $\mathbb{S}^2$.
\begin{lem} \label{lem:fibrations}
Let $\Sigma$ be a compact, connected orientable surface different from
$\mathbb{S}^2$.
\begin{enumerate}[(a)]
\item Suppose that $\Sigma$ has empty boundary.
\begin{enumerate}[(i)]
\item If $\Sigma$ is the $2$-torus $\mathbb{T}^2$ 
then~(\ref{eq:psequence}) splits for all $k,n\in\N$.
\item If $n=1$ then both~(\ref{eq:psequence}) and~(\ref{eq:sequence})
split for all $k\in\N$.
\item Let $n\geq 2$ and $k\in\N$. If $\Sigma\neq \mathbb{T}^2$,
then~(\ref{eq:psequence}) 
does not split. If further
$k=1$ then~(\ref{eq:sequence}) does not split.
\end{enumerate}
\item If $\Sigma$ has non-empty boundary then~(\ref{eq:psequence})
and~(\ref{eq:sequence}) split for all $k,n\in \mathbb{N}$.
\end{enumerate}
\end{lem}

\begin{proof}

\begin{enumerate}
\item
\begin{enumerate}
\item The statement is a consequence of~\cite{FaN} and the fact that
$\mathbb{T}^2$  admits a non-vanishing vector field.

\item Suppose that $n=1$. Then~(\ref{eq:psequence}) splits using
\cite{GG1} in the orientable case. The fact that the upper right-hand vertical arrow
$P_{1}(\Sigma) \to B_{1}(\Sigma)$ of the
diagram~(\ref{eq:commsequence1}) is the identity yields a section for
$B_{k, 1}(\Sigma) \to  B_1(\Sigma)$.

\item For~(\ref{eq:psequence}), this follows from~\cite{GG1}. Next,
let $k=1$, suppose that~(\ref{eq:sequence}) splits, and let $s$ be a
section for $\MB{1}{\Sigma} \to  B_n(\Sigma)$. For $m\in\N$, let
$\tau_{m}: B_{m}(\Sigma) \to S_{m}$ denote the usual permutation
homomorphism that appears in the short exact
sequence~(\ref{eq:permutation}). If $x\in P_{n}(\Sigma)$ then
$\tau_{n}(x)=1$, and so $\tau_{1+n}(s(x))=1$ by commutativity of the
diagram~(\ref{eq:commsequence}) and the fact that $k=1$ (by abuse of
notation, $\tau_{1+n}$ also denotes the restriction of $\tau_{1+n}$ to
$\MB{1}{\Sigma}$). Hence $s(x)\in P_{1+n}(\Sigma)$. Since
$P_{1+n}(\Sigma) \to P_{n}(\Sigma)$ is the restriction of
$\MB{1}{\Sigma} \to  B_n(\Sigma)$ to $P_{1+n}(\Sigma)$, the
restriction of $s$ to $P_{n}(\Sigma)$ gives rise to a section
for~(\ref{eq:psequence}), and so we obtain a contradiction.
\end{enumerate}
\item Suppose that $\Sigma$ has non-empty boundary, and let $C$ be a
boundary component of $\Sigma$. Then $\Sigma'=\Sigma\setminus C$ is
homeomorphic to a compact surface with a single point deleted. The
fact that $\Sigma$ and $\Sigma'$ are homotopy equivalent implies that
their configuration spaces are also homotopy equivalent, and hence the
pure braid groups (resp.\ braid groups, mixed braid groups) of
$\Sigma$ are isomorphic to the corresponding pure braid groups (resp.\
braid groups, mixed braid groups) of $\Sigma'$. Applying the methods
of~\cite{GG1}, the short exact sequences~(\ref{eq:psequence})
and~(\ref{eq:sequence}) split for $\Sigma'$, and so split for
$\Sigma$.\qedhere
\end{enumerate}
\end{proof}

Taking into account the above 
discussion concerning the existence of braid group representations via
the induced action on first homology, this lemma indicates that one
might use~(\ref{eq:sequence}) to look for representations of surface
braid groups in the case where the boundary is non empty.



\section{Lower central series  for surface braid groups} \label{section4}



\subsection{Definitions and known results}


In this section we will give some results on the lower central series of (mixed) surface braid groups that will turn out useful in the
next section when we come to study their representations. Given a group $G$, we recall that the \emph{lower central series} of $G$
is the filtration $G =\Gamma_1(G) \supseteq \Gamma_2(G) \supseteq \cdots$, 
where $\Gamma_i(G)=[G,\Gamma_{i-1}(G)]$ for $i\geq 2$. The group $G$ is said to be
\emph{perfect} if $G=\Gamma_2(G)$. 
Following P.~Hall, for a
group-theoretic property $\mathcal{P}$, a group $G$ is said to be
\emph{residually $\mathcal{P}$} if for any (non-trivial) element $x$
in $G$, there exists a group $H$ with the property  $\mathcal{P}$ and
a surjective homomorphism $\phi: G \to H$ such that  $\phi(x) \not=1$. It is well known that a group $G$
is residually nilpotent if and only if
$\bigcap_{i \ge 1}\Gamma_i(G)=\{ 1\}$. 

In what follows, we denote a compact, connected orientable surface
with one boundary component of positive genus $g$ by $\surf{g}{}$. The surface
 $\surf{g}{} \setminus \sset{x_{1}}{x_{n}}$ will be denoted by $\surf{g}{, n}$. We will focus on  $\surf{g}{}$ since, according to \relem{fibrations}, for these particular surfaces the sequence (MSB) splits.
The following result is well known
(see~\cite{BGG, GL} for instance).

\begin{prop} \label{braidlcs}
Let $B_n$ be the Artin braid group on $n\ge 3$ strands. Then $\Gamma_1(B_n)/\Gamma_2(B_n) \cong \Z$ and $\Gamma_2(B_n)=\Gamma_3(B_n)$.
\end{prop} 

The case of braid groups of orientable surfaces of genus at least one, 
is much richer (the cases of the sphere $\mathbb{S}^2$ and the punctured sphere were studied in~\cite{GG2}): related statements  for $\surf{g}{}$ may be summarised as follows (see~\cite{BB2, BGG, GL} for similar results for the other orientable surfaces).

\begin{thm}(\cite{BGG})\label{th:ggeq2}
Let $g\geq 1$ and $n\geq 3$. Then:
\begin{enumerate}[(a)]
\item\label{it:g1sigma} $\Gamma_1(B_n(\surf{g}{}))/ \Gamma_2(B_n(\surf{g}{}))=\Z^{2g} \oplus \Z_2$.
\item\label{it:g2sigma} $\Gamma_2(B_n(\surf{g}{}))/ \Gamma_3(B_n(\surf{g}{}))=\Z$.
\item\label{it:g3sigma} $\Gamma_3(B_n(\surf{g}{}))= \Gamma_4(B_n(\surf{g}{}))$. Moreover
$\Gamma_3(B_n(\surf{g}{}))$ is perfect for $n\ge 5$.
\item\label{it:resid} $B_n(\surf{g}{})$ is not residually nilpotent.
\end{enumerate}
\end{thm}

 
\subsection{Group presentations for surface braid groups and their $3$-commutator groups}
 

\begin{thm}(\cite{BGG})\label{th:presbng}
Let $n\geq 1$. The group~$\bsurf{n}{g}{}$ admits the following group presentation:

\noindent \textbf{Generators:}  $a_1, b_1, \ldots, a_g, b_g, \sigma_1, \ldots, \sigma_{n-1}$.

\noindent \textbf{Relations:} 
\begin{gather}
\text{$\sigma_i\sigma_j=\sigma_j\sigma_i$ if $\lvert i-j \rvert \geq
2$}\label{eq:artin1}\\
\text{$\sigma_i\sigma_{i+1}\sigma_i= \sigma_{i+1}\sigma_i \sigma_{i+1}$
for all $1\leq i\leq n-2$}\label{eq:artin2}\\
\text{$c_i\sigma_j= \sigma_j c_i$ for all $j\geq 2$,   $c_i=a_i$ or
  $b_i$ and $i=1, \ldots, g$}\label{eq:asjg}\\
\text{$c_i \sigma_1 c_i \sigma_1= \sigma_1 c_i \sigma_1 c_i$   for
  $c_i=a_i$ or $b_i$ and $i=1, \ldots, g$}\label{eq:bsbg}\\
\text{$a_i \sigma_1 b_i = \sigma_1 b_i \sigma_1 a_i \sigma_1$ for  $i=1, \ldots, g$}\label{eq:abbag}\\
\text{$c_i \sigma_1^{-1}  c_j \sigma_1=\sigma_1^{-1} c_j \sigma_1 c_i$
  for $c_i=a_i$ or $b_i$, $c_j=a_j$ or $b_j$ and $1\le j<i\le g$}\label{eq:cddcg}
\end{gather}
\end{thm}

In Figure 1 we recall a geometric interpretation of the generators of $\bsurf{n}{g}{}$;
we represent  $\surf{g}{}$  as a polygon  with $4g$ sides, equipped with the standard identification of edges and one boundary component.
We may consider braids  as paths on the polygon,
which we draw with the usual `over and under' information at  the crossing points.
For the braid $a_i$ (respectively $b_j$), the only non-trivial string is the first one,
which passes  through the wall $\alpha_i$ (respectively the wall $\beta_j$).
The braids $\sigma_1, \dots, \sigma_{n-1}$ are the standard braid generators of the disc. One can easily write a braid represented by a loop of the first strand around the boundary component as the composition of the generators (see for instance Section~2.2 of \cite{B}). 

\begin{figure}[h]
 \centering
 \includegraphics[width=10cm]{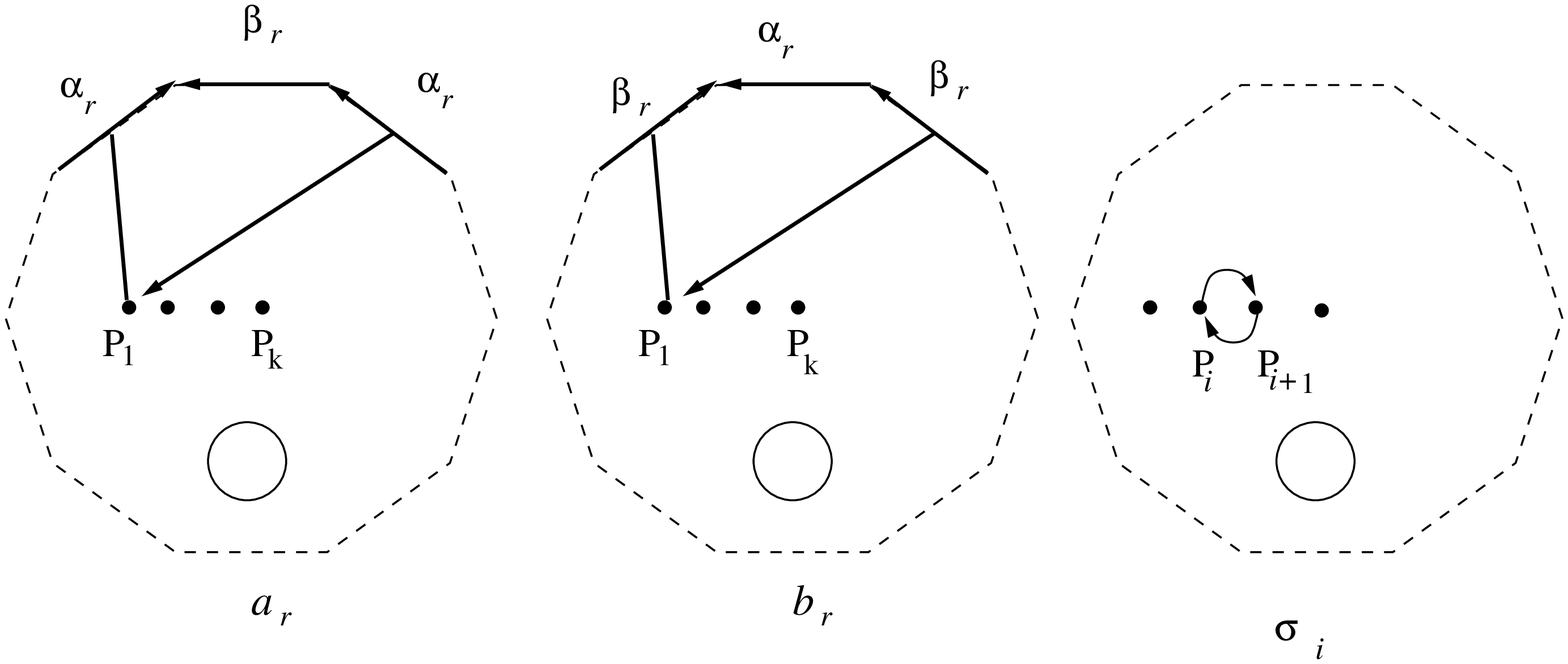}
 \caption{The generators $\si_1, \ldots, \si_{k-1}, a_1, b_1, \ldots, a_g, b_g$}
\end{figure}

From Theorem~\ref{th:presbng} one may deduce the following result.

\begin{lem}(\cite{BGG})\label{lem:presbng} 
Let $n\ge 3$.
The quotient group~$\bsurf{n}{g}{}/\Gamma_3(\bsurf{n}{g}{})$ admits the following group presentation:

\noindent \textbf{Generators:}

$a_1, b_1, \ldots, a_g, b_g, \sigma$.

\noindent  \textbf{Relations:}
\begin{gather}
\text{$a_1,b_1, \ldots, a_g, b_g$ and $\sigma$ commute pairwise except  $(a_i, b_i)_{i=1,
\ldots, g}$};\\
\text{$[a_1,  b_1]=\cdots=   [a_g,  b_g] =\sigma^2$}.
\end{gather}
\end{lem}

\medskip

We point out that the corresponding result proved in~\cite[Theorem~1 ]{BGG} is in fact for closed surfaces but  the proof given there  can be repeated verbatim to prove Lemma \ref{lem:presbng}.

The following Corollary is a straightforward consequence  of Lemma \ref{lem:presbng}.

\begin{cor}\label{cor:decgamma}
Let $n\ge 3$. The group~ $\bsurf{n}{g}{}/\Gamma_3(\bsurf{n}{g}{}) $ is isomorphic to the semi-direct product
$$
G_g=\left( \Z \times \Z^g \right) \rtimes \Z ^g.
$$
More precisely, the first factor $\Z$ is central and is generated by $ \sigma $, the second factor $\Z^g$
is generated by $\{ a_1, \ldots, a_g\}$, and the third factor $\Z^g$ is generated by $\{b_1, \ldots, b_g\}$.
Any generator 
 $b_j$ (for $1 \le j \le g$) acts trivially on $a_ 1, \dots,  a_{j-1}, a_{j+1}$ and 
 $  b_j a_j b_j^{-1} =\sigma^{-2} a_j$.   Hence, $\bsurf{n}{g}{}/\Gamma_3(\bsurf{n}{g}{})$ is a central extension of $\Z^{2g}$ by $\Z$.
\end{cor}

Thus every element of $B_n(\Sigma_g)/\Gamma_3(B_n(\Sigma_g))$ can be written  in a unique way in the form
$ \sigma^p  \prod_{i=1}^g a_i^{m_i} \prod_{i=1}^gb_i^{n_i}$.


\subsection{Group presentations for mixed braid groups of surfaces and their $3$-commutator groups}
 

Following the standard definition of a Coxeter system, we introduce the notion of a \emph{surface braid group system}. 
\begin{defin} Let $G$ be a group. Let $S = \{\si_1,\cdots,\si_{k-1}\}$, $AB= \{a_1,\cdots, a_g,b_1,$ $\cdots, b_g\}$ and $Z = \{\zeta_1,\ldots,\zeta_n\}$ be subsets of $G$ such that $k\geq 1$ and $g,n\geq 0$. We say that $(G,S,AB,Z)$ is a \emph{surface braid group system} if $G$ admits the following group presentation: 
$$\begin{array}{ll}
\si_i\si_j=\si_j\si_i;&\lvert i-j \rvert \geq
2\\
\si_i\si_{i+1}\si_i= \si_{i+1}\si_i \si_{i+1};& 1\leq i\leq k-2\\
a_i\si_j= \si_j a_i\textrm{ and }b_i\si_j= \si_j b_i;&j\in\{2,\ldots,k-1\},\ i\in \{1, \ldots, g\}\\
a_i \si_1 a_i \si_1= \sigma_1 a_i \sigma_1 a_i\textrm{ and } b_i \sigma_1 b_i \sigma_1= \sigma_1 b_i \sigma_1 b_i &i\in \{1, \ldots, g\} \\
a_i \sigma_1 b_i = \sigma_1 b_i \sigma_1 a_i \sigma_1;& i\in \{1, \ldots, g\}\\
c_i (\sigma_1^{-1}  c_j \sigma_1)=(\sigma_1^{-1} c_j \sigma_1) c_i; &c_i\in\{a_i,b_i\}, c_j\in\{a_j,b_j\},\   j<i\\
\zeta_j\si_i = \si_i\zeta_j&i\neq 1;\\
(\si_1^{-1}\zeta_j\si_1)a_\ell = a_\ell (\si_1^{-1}\zeta_j\si_1)\\(\si_1^{-1}\zeta_j\si_1)b_\ell = b_\ell (\si_1^{-1}\zeta_j\si_1);\\
(\si_1^{-1}\zeta_j\si_1)\zeta_\ell = \zeta_\ell (\si_1^{-1}\zeta_j\si_1)& j<\ell; \\
(\si_1 \zeta_j \si_1)\zeta_j = \zeta_j(\si_1 \zeta_j \si_1).\\
\end{array}$$
\end{defin}

The following Lemma is a straightforward consequence of the group presentations given in \cite{B}.

\begin{lem}
\begin{enumerate}[(i)]
\item There exist $S = \{\si_1,\cdots,\si_n\}$ and $AB= \{a_1,\cdots, a_g,b_1,\cdots, b_g\}$ such that $(\bsurf{n}{g}{},S,AB,\emptyset)$ is a surface braid group system.
\item There exist $S = \{\si_1,\cdots,\si_k\}$, $AB= \{a_1,\cdots, a_g,b_1,\cdots, b_g\}$  and $Z = \{\zeta_1,\ldots,$ $\zeta_n\}$ such that $(\bsurf{k}{g}{,n},S,AB,Z)$ is a surface braid group system. 
\end{enumerate}
\end{lem}

\begin{prop}\label{prop:presMB}
The group~$\bsurf{k,n}{g}{}$ admits the following group presentation:

\noindent \textbf{Generating set:} $S\cup \tilde{S}\cup AB\cup \widetilde{AB}\cup Z$ with 
$$\begin{array}{lcllcl} \tilde{S}&=&\{\tsi_1, \ldots, \tsi_{n-1}\},&S&=&\{ \si_1, \ldots, \si_{k-1}\},\\
\widetilde{AB}&=&\{\tia_1, \tib_1, \ldots, \tia_g, \tib_g\},&AB&=&\{a_1, b_1, \ldots, a_g, b_g\},\\
Z&=&\{\zeta_1,\cdots, \zeta_n\}.
\end{array}$$

\noindent  \textbf{Relations:}
 \begin{enumerate}[(a)]
\item the relations associated with the system~$(\bsurf{k}{g}{,n},S,AB,Z)$;
\item the relations associated with the system~$(\bsurf{n}{g}{},\tilde{S},\widetilde{AB},\emptyset)$;
\item the relations describing the action of~$\bsurf{n}{g}{}$ on~$\bsurf{k}{g}{,n}$;
\begin{enumerate}
\item $\tsi_i\si_j\tsi_i^{-1} = \tia_i\si_j\tia_i^{-1} = \tib_i\si_j\tib_i^{-1} =\si_j$;

\item  $\tsi_ia_j\tsi_i^{-1} = a_j$ ; $\tsi_ib_j\tsi_i^{-1} = b_j$ 

\item $\left\{\begin{array}{lll}\tsi_i\z_{i+1}\tsi_i^{-1} = \z_i;\\
\tsi_i\z_{i}\tsi_i^{-1} = \z_i^{-1}\z_{i+1}\z_i;\\
 \tsi_i\z_{j}\tsi_i^{-1} = \z_j,& j\neq i,i+1;\end{array}\right.$
$\left\{\begin{array}{llll}\tia_i \z_1 \tia_i^{-1} =\z_1^{a_i \z_1} ;\\  
\tib_i \z_1 \tib_i^{-1} =  \z_1^{b_i \z_1};\\  
\tia_i \z_j \tia_i^{-1} = \z_j^{[a_i^{-1},\z_1^{-1}]}  &j\neq 1;\\ 
\tib_i \z_j \tib_i^{-1} = \z_j^{[b_i^{-1},\z_1^{-1}]}  &j\neq 1;
\end{array}\right.$

\item $\left\{\begin{array}{ll} \tia_ia_i\tia_i^{-1} = \z_1^{-1} a_i \z_1;\\ 
\tia_i a_j\tia_i^{-1} =  a_j^{[a_i^{-1},\z_1^{-1}]}&i > j;\\  
\tia_i a_\ell \tia_i^{-1} = a_\ell,&\ell > i;\end{array}\right.$  
$\left\{\begin{array}{ll}  \tib_i b_i\tib_i^{-1} = \z_1^{-1} b_i \z_1;\\ 
\tib_i b_j\tib_i^{-1} =  b_j^{[b_i^{-1},\z_1^{-1}]} &i > j;\\ 
\tib_i b_\ell \tib_i^{-1} = b_\ell,&\ell > i;\end{array}\right.$

\item   $\left\{\begin{array}{ll} \tia_ib_i\tia_i^{-1} = b_i \z_1;\\  
\tia_i b_j\tia_i^{-1} = b_j^{[a_i^{-1},\z_1^{-1}]}  &i > j;\\  
\tia_i b_\ell \tia_i^{-1} = b_\ell,&\ell > i
\end{array}\right.$ 
$\left\{\begin{array}{ll} 
\tib_ia_i\tib_i^{-1} = \z_1^{-1} a_i [b_i^{-1},\z_1^{-1}];\\
\tib_i a_j\tib_i^{-1} = a_j^{[b_i^{-1},\z_1^{-1}]}&i > j;\\  
\tib_i a_\ell \tib_i^{-1} = a_\ell,&\ell > i; \end{array}\right.$
\end{enumerate}
where $a^b:=b^{-1}ab$.
\end{enumerate}
\end{prop}

\begin{proof} As we recalled in Section~2, the short exact sequence~(MSB):
$1 \to \bsurf{k}{g}{,n} $ $\to \bsurf{k,n}{g}{} $ $\to \bsurf{n}{g}{} \to 1$ 
splits. Therefore, the group~$\bsurf{k,n}{g}{}$ is isomorphic to $\bsurf{n}{g}{}\ltimes \bsurf{k}{g}{,n}$. We may interpret the braids depicted in Figure~2 as geometric representatives of generators of $\bsurf{k}{g}{,n}$, and those depicted in Figure~3 as the coset representatives of generators of
$\bsurf{n}{g}{}$ in $\bsurf{n}{g}{}$.   The result follows by a straightforward verification of the corresponding geometric braids, see for instance Figure~4. 
\end{proof}

\begin{figure}[h]
 \centering
 \includegraphics[width=10cm,]{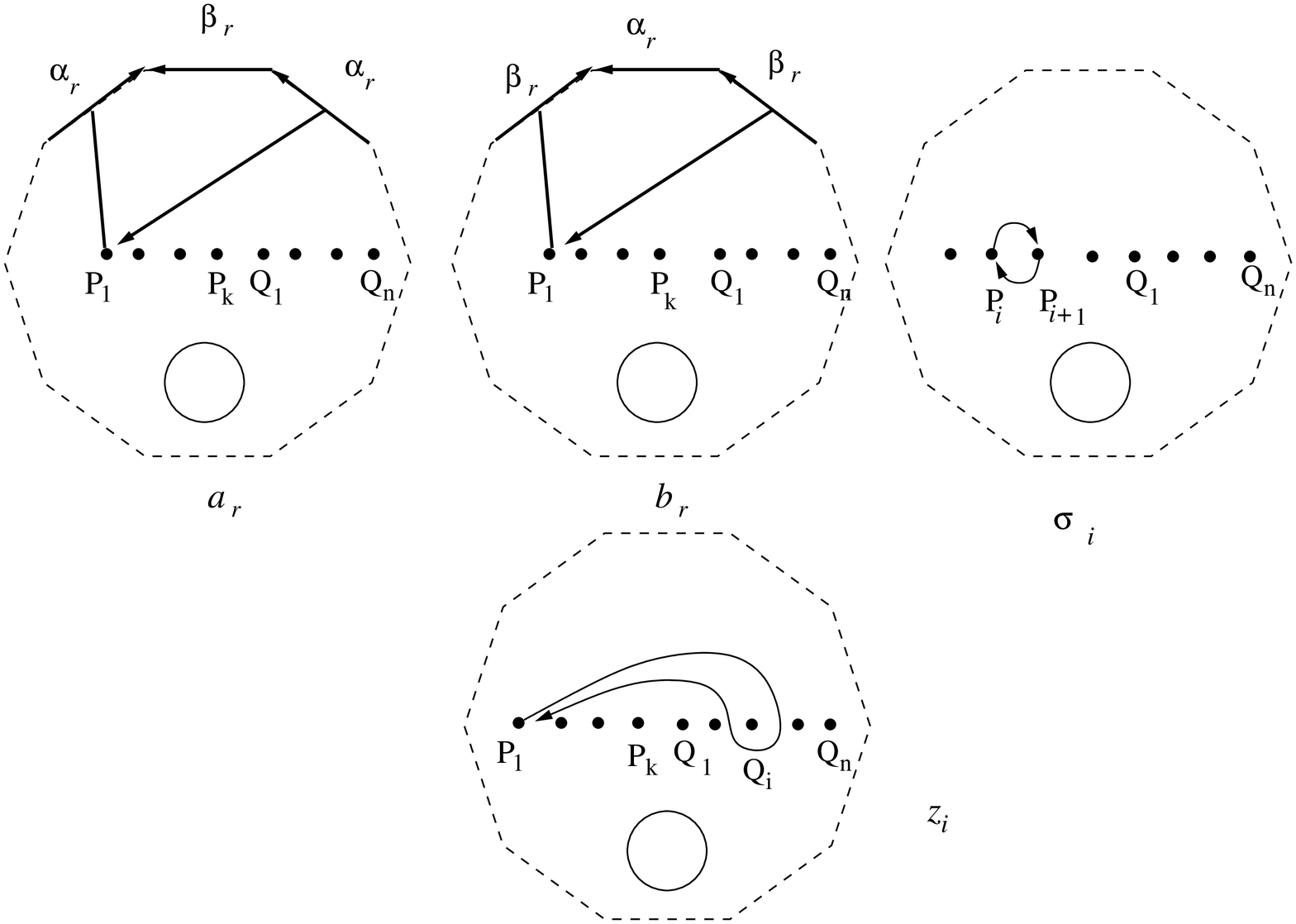}
 \caption{The generators $\si_1, \ldots, \si_{k-1}, a_1, b_1, \ldots, a_g, b_g,\zeta_1,\cdots, \zeta_n$}
\end{figure}

\begin{figure}[h]
 \centering
 \includegraphics[width=10cm]{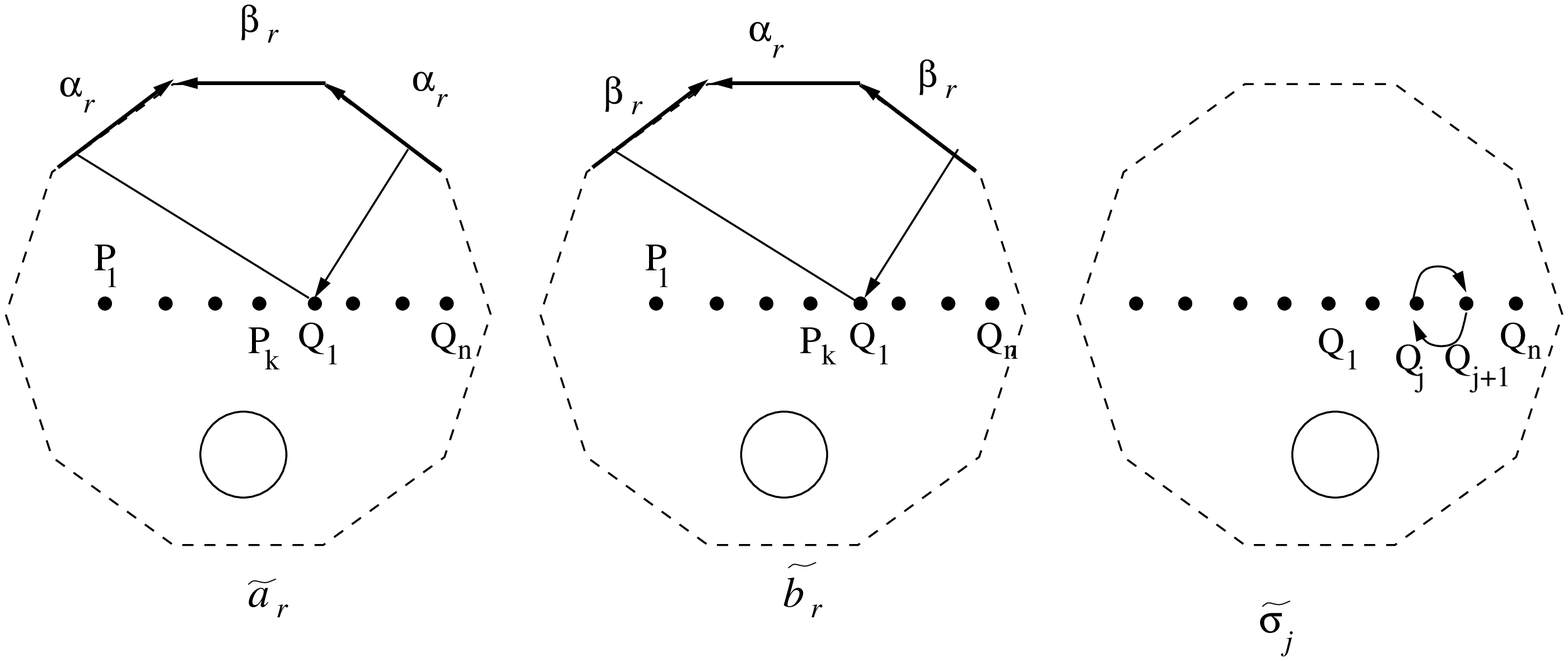}
 \caption{The generators $\tsi_1, \ldots, \tsi_{n-1}, \tia_1, \tib_1, \ldots, \tia_g, \tib_g$}
\end{figure}

\begin{figure}
 \centering
 \includegraphics[width=10cm]{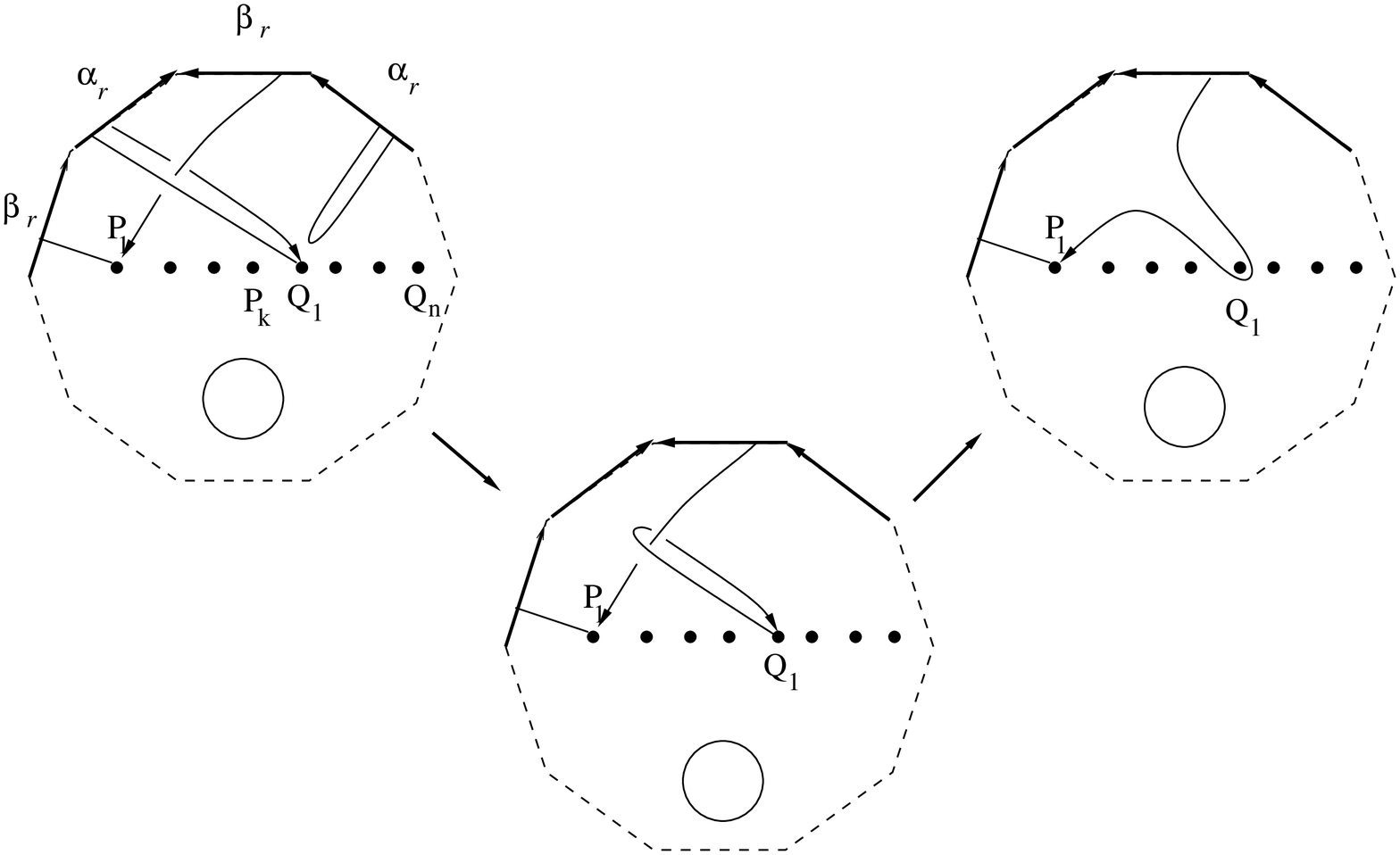}
 \caption{The braids $\tia_ib_i\tia_i^{-1}$ and  $b_i \z_1$ are isotopic.}
\end{figure}

As in the case of $ B_n(\surf{g}{})/\Gamma_3(B_n(\surf{g}{}))$ (Corollary~\ref{cor:decgamma}), we may obtain a group presentation of  $\bsurf{k,n}{g}{}/\Gamma_3(\bsurf{k,n}{g}{})$ from the previous proposition, and decompose the group as a semi-direct product.

\begin{lem}\label{lem:presgamma} Let  $k,n\geq 3$. Then the group $\bsurf{k,n}{g}{}/\Gamma_3(\bsurf{k,n}{g}{})$ has the following group presentation:

\noindent \textbf{Generators:} $\si, \tsi, \zeta, a_1, b_1, \ldots, a_g, b_g, \tia_1, \tib_1, \ldots, \tia_g, \tib_g .$\\
\textbf{Relations:} 
\begin{enumerate}
\item $[\si,a_i] = [\si,b_i] = [\tsi,\tia_i] = [\tsi,\tib_i] = [\tsi,a_i] = [\tsi,b_i] = [\si,\tia_i] = [\si,\tib_i] = [\si,\tsi] =  1$;
\item $[a_i,a_j] = [a_i,b_j] = 1$ for $i\neq j$ and  $[a_j,b_j] =\si^2$; 
\item $[\tia_i,\tia_j] = [\tia_i,\tib_j] = 1$ for $i\neq j$  and  $[\tia_j,\tib_j] = \tsi^2$
\item $[a_i,\tia_j] = [b_i,\tib_j] = 1$ for any pair $(i,j)$;
\item $[b_i,\tia_j] = [\tib_j,a_i] = 1$ for $i\neq j$  and  $[b_i,\tia_i] = [\tib_i,a_i]= \zeta$;
\item $[\zeta,a_i] = [\zeta,b_i] = [\zeta,\tia_i] = [\zeta,\tib_i] = [\zeta,\si] = [\zeta,\tsi] = 1$.
\end{enumerate}
\end{lem}

\begin{proof} Consider the group presentation of~$\bsurf{k,n}{g}{}$ given in Proposition~\ref{prop:presMB} and the map $p :\bsurf{k,n}{g}{}\to \bsurf{k,n}{g}{}/\Gamma_3(\bsurf{k,n}{g}{})$.
It follows from the braid relations $\si_i\si_{i+1}\si_i = \si_{i+1}\si_{i}\si_{i+1}$ and $\tsi_i\tsi_{i+1}\tsi_i = \tsi_{i+1}\tsi_{i}\tsi_{i+1}$ that $\si_i = \si_{i+1}[\si_i,\si_{i+1}]$  and  
$\tsi_i = \tsi_{i+1}[\tsi_i,\tsi_{i+1}]$. But  $[\tsi_i,\tsi_{i+1}]$ and $[\si_i,\si_{i+1}]$ belong to~$\Gamma_2(\bsurf{k,n}{g}{})$. 
Their images under $p$  therefore belong to  the centre of  $\bsurf{k,n}{g}{}/\Gamma_3(\bsurf{k,n}{g}{})$. Since  $p(\si_{i})$ and  $p(\si_{i+1})$  are center elements,  the equality~$p(\si_i \si_{i+1} \si_i) = p(\si_{i+1}\si_{i}\si_{i+1})$ implies  that  $p(\si_i) = p(\si_{i+1})$. We denote the image of $\si_i$ under $p$ by $\si$. 
Similarly, the $\tsi_j$ have the same image under $p$, which we denote by $\tsi$ in what follows. By abuse of notation, we also denote $p(a_i), p(b_i), p(\tia_i)$ and $p(\tib_i)$ by  $a_i, b_i, \tia_i$ and $\tib_i$ respectively. Using the fact that $\si_ia_j = a_j\si_i$ for $i\neq j$, we 
obtain $\sigma a_j = a_j\sigma$ for all $j$. Similarly, we have $\si b_j = b_j\si$, $\si \tia_j = \tia_j\si$, $\tsi \tib_j = \tib_j\tsi$ 
and $\tsi p(\zeta_j) = p(\zeta_j)\tsi$. We deduce from relations~$a_i (\sigma_1^{-1}  a_j \sigma_1)=(\sigma_1^{-1} a_j \sigma_1) a_i$ for $i\neq j$ that $a_ia_j = a_ja_i$. Similarly we have $a_ib_j = b_ja_i$, $\tia_i\tia_j = \tia_j\tia_i$ and $\tib_i\tia_j = \tia_j\tib_i$ for $i\neq j$. For the same reason, we have $\tia_i p(\zeta_j) = p(\zeta_j)\tia_i$ and $\tib_i p(\zeta_j) = p(\zeta_j)\tib_i$. Now, from the equality~$a_j \si b_j = \si b_j \si a_j \si$, we deduce the relation $[a_j,b_j] =\si^2$. Similarly we have~$[\tia_j,\tib_j] =\tsi^2$. From Relations~$(i)$ and~$(ii)$ of the presentation of~$\bsurf{k,n}{g}{}$, we have~$[\tsi,a_i] = [\tsi,b_i] = [\si,\tia_i] = [\si,\tib_i] = [\si,\tsi] =  1$. From Relations~$(iii)$ we see that $\si p(\zeta_{i+1})\si^{-1} = \zeta_i = \si_j\zeta_{i}\si_j^{-1}$ with $j\neq i,i+1$. Then, all the $\zeta_i$ have the same image under $p$, which we denote by $\zeta$. 
Then $[\zeta,\tia_i] = [\zeta,\tib_i] = [\zeta,\si] = [\zeta,\tsi] = 1$. Still using Relations~$(iii)$, we have $[\zeta,a_i] = [\zeta,b_i] = 1$. From Relations~$(iv)$ we obtain $[a_i,\tia_j] = [b_i,\tib_j] = 1$, and from Relations~$(v)$ we get $[b_i,\tia_j] = [\tib_j,a_i] = 1$  for $i\neq j$  and  $[b_i,\tia_i] = [\tib_i,a_i]= \zeta$. Then, the defining relations of the presentation given in the lemma hold in $\bsurf{k,n}{g}{}/\Gamma_3(\bsurf{k,n}{g}{})$ if one takes $\si = p(\si_1), \tsi = p(\tsi_1), \zeta = p(\zeta_1), a_i = p(a_i), b_i = p(b_i), \tia_i  = p(\tia_i)$ and $\tib_i = p(\tib_i)$. Conversely, let $G$ be the group defined by the presentation given in the lemma. We may check without difficulty that we have a group homomorphism from $\bsurf{k,n}{g}{}$ to $G$ that sends $\si_i,\tsi_j,a_{i'},b_{i''}, \tia_{j'},\tia_{j''}$ and $\zeta_\ell$ to $\si,\tsi,a_{i'},b_{i''}, \tia_{j'},\tia_{j''}$ and $\zeta$ respectively.  In order to prove that $G$ is isomorphic to $\bsurf{k,n}{g}{}/\Gamma_3(\bsurf{k,n}{g}{})$, we need to prove that in $G$, the equality $[a,[b,c]] = 1$ holds for all $a,b,c$ in $G$. Using the equality~$[a,bc] = [a,b]b[a,c]b^{-1}$, it is enough to consider the case where $a,b$ and $c$ are defining generators of $G$. But in these particular cases, the equality clearly holds since $\si,\tsi$ and $\zeta$ belong to the centre of $G$, and therefore any commutator~$[x,y]$, where $x,y$ are defining generators, is in the centre of $G$. Thus $G$ and $\bsurf{k,n}{g}{}/\Gamma_3(\bsurf{k,n}{g}{})$ are isomorphic.  \end{proof}

\begin{cor}\label{cor:decgammamix}
Let  $k,n\geq 3$. The group $\displaystyle \frac{\bsurf{k,n}{g}{}}{\Gamma_3(\bsurf{k,n}{g}{})}$ is isomorphic to the semi-direct product:
$$
H_g:=\left(\Z^3  \times \Z^{2g} \right) \rtimes \Z^{2g}.
$$
More precisely, the first factor $\Z^3$  is central and is generated by $ \sigma, \tsi, \zeta$, the second factor $\Z^{2g}$
is generated by $\{ a_1, \ldots, a_g, \tia_1 \ldots,  \tia_g \}$, and the third factor $\Z^{2g}$ is generated by 
$\{ b_1, \ldots, b_g, \tib_1 \ldots,  \tib_g \}$.  \end{cor}

\noindent The actions defining the above semi-direct product are as described in Lemma~\ref{lem:presgamma}. Then 
$\bsurf{k,n}{g}{}/\Gamma_3(\bsurf{k,n}{g}{})$ can be seen as a central extension of $\Z^{4g}$ by $\Z^3$.
We remark that the action of any generator on the previous ones is trivial except in at most two cases: for instance the only non-trivial actions of
$\tia_j$ are $\tia_j  \tib_j \tia_j^{-1}=\tsi^{2} \tib_j$ and $\tia_j  b_j \tia_j^{-1}=\zeta^{-2}b_j$. 
Finally, notice that every element $w$ of $\bsurf{k,n}{g}{}/\Gamma_3(\bsurf{k,n}{g}{})$ may be written  in a unique way in the form
$w= \sigma^p \tsi^q \zeta^r  \prod_{i=1}^g a_i^{m_i}  \tia_i^{M_i} \prod_{i=1}^g b_i^{n_i}  \tib_i^{N_i}$.


\section{Mixed surface braid sequences and representations of  surface braid groups}  \label{section5}


As in previous sections, we will focus on connected orientable surfaces  of positive genus
with a single  boundary component. 
One might ask whether it is possible to use the short exact sequence $(MSB)$ to obtain representations of $B_n(\surf{g}{})$ which are extensions
of Bigelow-Krammer-Lawrence representations of $B_n$. First, let us denote by $\iota_n: B_n \to B_n(\surf{g}{})$ and $\iota_{k,n}: B_k(\D_n) \to B_k(\surf{g}{, n})$
the homomorphisms induced by the inclusion of $\D^2$ into $\surf{g}{}$ and of $\D_n$ into $\surf{g}{, n}$ respectively. We recall that they are injective  (see for instance \cite{PR}). 
  
  \begin{rem}\label{rem_embed}
  In particular $B_k(\D_n)$ is generated as a subgroup of  $B_k(\surf{g}{, n})$ by $\sigma_1,\ldots,\sigma_{k-1}$
  and $\zeta_1, \ldots, \zeta_{n}$ (see for instance \cite{BB, B}).
  \end{rem}
  
Given $\beta \in B_n$, we denote the induced action by conjugation of $\beta$ on $B_{k}(\D_n)$, and abusing notation, 
the action of $\beta$ (regarded as an element of $B_n(\surf{g}{})$) on $B_k(\surf{g}{, n})$, by $\beta_*$.
As in Section~2, for the case $k=1$  we consider the length function $p_1: B_1(\D_n) \to G_1=\Z=\langle  t \rangle$, while for $k>1$ we set  $G_k$ to be  the group $\Z^2=\langle q, t\rangle$.
The corresponding homomorphism $p_k: B_k(\D_n) \to G_k$ for $k>1$  sends $\sigma_1, \ldots, \sigma_k$ to $q$ and $\zeta_1, \ldots, \zeta_n$ to $t$.
As we mentioned in Section~2, the fact that  the action of $B_n$ on 
$B_k(\D_n)$ commutes with $p_k:  B_k(\D_n) \to G_k$ implies that  $B_n$ acts on a regular covering.
For $k=1$ the induced action on homology gives rise to  the Burau representation, while for $k>1$ we obtain faithful linear representations of $B_n$.

 \begin{defin} \label{fonddef}
  Let  $\qk: B_k(\surf{g}{, n}) \to G_k(\surf{g}{})$ be a surjective homomorphism. We will say that the homomorphism $\qk$ is a lifting extension of the map $p_k:  B_{k}(\D_n)
  \to G_k$ defined above
  if there exists an injective homomorphism  $\bar\iota_k: G_k \to G_k(\surf{g}{})$ such that 
  $\qk\circ \iota_{k,n}= \bar\iota_k \circ p_k$,
  and if further for all  $\beta \in B_n(\surf{g}{})$ there exists a homomorphism $\bar\beta_*: G_k(\surf{g}{}) \to G_k(\surf{g}{})$
  such that $\bar\beta_* \circ \qk= \qk  \circ \beta_* $. 
  \end{defin}
  
  The second condition means that there is an induced  action of $B_n(\surf{g}{})$ on the homology of the covering space of $\FF_k(\surf{g}{, n})/S_k$.
  In other words, the surjection $\qk$ is a lifting extension of $p_k$ if the following diagram commutes for any $\beta \in B_n$:
\begin{equation} \label{eq:commsequence}
\begin{xy}*!C\xybox{%
\xymatrix{%
 G_k  \ar@{=}[ddd]^{Id} \ar@{^{(}->}[rrr]^{\bar\iota_k}  &  & & G_k(\surf{g}{}) \ar[ddd]^{\bar\beta_*}\\ 
& B_{k}(\D_n) \ar@{->>}[lu]^{p_k}   \ar@{^{(}->}[r] \ar[d]^{\beta_*} & B_k(\surf{g}{, n}) \ar@{->>}[ru]^{\qk}  \ar[d]^{\beta_*} & \\
 & B_{k}(\D_n) \ar@{->>}[ld]^{p_k}   \ar@{^{(}->}[r]  & B_k(\surf{g}{, n}) \ar@{->>}[rd]^{\qk}   & \\  
 G_k     \ar@{^{(}->}[rrr]^{\bar\iota_k} &  &  &G_k(\surf{g}{})}}
\end{xy}
\end{equation}
We remark that the middle homology group $H_k^{BM} \left( \displaystyle \frac{\widetilde{\FF_k(\surf{g}{, n})}}{S_k} \right)$ of the covering space   $\displaystyle \frac{\widetilde{\FF_k(\surf{g}{, n})}}{S_k}$
of $\displaystyle \frac{\FF_k(\surf{g}{, n})}{S_k}$ is a free $\Z[G_k(\surf{g}{})]$-module (see Lemma~3.3
of \cite{HK}) and that some $\beta \in  B_n(\surf{g}{})$ acts as a $\Z[G_k(\surf{g}{})]$-module morphism of  $H_k^{BM}(\widetilde{\FF_k(\surf{g}{, n})/S_k})$
if and only if the  map $\bar\beta_*: G_k(\surf{g}{}) \to G_k(\surf{g}{})$ defined above is the identity homomorphism (see Section~2 of \cite{HK}).
If this property holds for all $\beta \in  B_n(\surf{g}{})$ we thus obtain a representation of $B_n(\surf{g}{})$ in $Aut_{\Z[G_k(\surf{g}{})]}(H_k^{BM}\left( \displaystyle \frac{\widetilde{\FF_k(\surf{g}{, n})}}{S_k} \right) )$ which is a linear representation if $G_k(\surf{g}{})$ is Abelian. The above discussion gives rise naturally to the following definition, which provides a notion of extension of the
Bigelow-Krammer-Lawrence representations  from  $B_n$ to $B_n(\surf{g}{})$.

\begin{defin}
  Let  $\qk: B_k(\surf{g}{, n})  \to G_k(\surf{g}{})$ be  a lifting extension of the homomorphism $p_k:  B_{k}(\D_n)
  \to G_k$. The homomorphism  $\qk$ is said to be a linear extension of $p_k$ if for any $\beta \in B_n(\surf{g}{})$, we have that  $\bar\beta_*=Id$
  and that $G_k(\surf{g}{})$ is Abelian. 
 \end{defin}

The following proposition states that it is not in fact possible to extend the Bigelow-Krammer-Lawrence representations from 
$B_n$ to $B_n(\surf{g}{})$.

\begin{prop} \label{nolinearext}
There is no homomorphism $\qk: B_k(\surf{g}{, n})  \to G_k(\surf{g}{})$ that is a linear extension of $p_k:  B_{k}(\D_n)
  \to G_k$. 
\end{prop}
\begin{proof}
This result is a reformulation of Lemma~2.6 of \cite{HK}. We sketch the proof. The action by conjugation of $B_n(\surf{g}{})$
on $B_k(\surf{g}{, n})$  is described in \repr{presMB}, where  $B_k(\surf{g}{, n})$
is the subgroup of $B_{k,n}(\surf{g}{})$ generated by $S, AB$ and $Z$. For any generator $g$, let us denote  its image in $G_k(\surf{g}{})$ by $[g]$,
and we consider $\tib_j$ as a generator of $B_n(\surf{g}{})$.
Since $\qk$ is a lifting extension of $p_k$, the action of $\tib_j$ by conjugation induces  the   homomorphism  $\overline{(\tib_j)_*}: G_k(\surf{g}{}) \to G_k(\surf{g}{})$, which sends the element $[a_i]$ to  $[a_i][\zeta_1]$.  Since  $\qk$ is also a linear extension,  $\overline{(\tib_j)_*}$ must coincide with the identity. One deduces that $[\zeta_1]=1$, but this cannot be true since  the hypothesis that $\qk$ is a lifting extension of $p_k$ implies that $[\zeta_1]=q$.
\end{proof}

In what follows we show that for $k\ge3$, there is a unique group  $G_k(\surf{g}{})$ such that one may define 
 a lifting extension  $\qk: B_k(\surf{g}{, n})  \to G_k(\surf{g}{})$ 
 of $p_k$.  We first characterise  the group $G_k(\surf{g}{})$ 
 which was defined in Section~3.1 of \cite{HK} (with the notation $G_\Sigma$) in terms of lower central series. We then show that this group
 is the unique group admitting a lifting extension for $p_k$, and we construct the corresponding homomorphism 
 $\qk: B_k(\surf{g}{, n})  \to G_k(\surf{g}{})$.

Consider the following diagram:
\begin{equation}\label{diagramexactsequence}
\begin{xy}*!C\xybox{%
\xymatrix{
1 \ar[r]  &  B_k(\surf{g}{, n})  \ar[d]^{\bqk}   \ar[r]    & B_{k,n}(\surf{g}{})  \ar[d]^{r_{k,n}} \ar[r]          &  B_n(\surf{g}{})   \ar[d]^{r_{n}} \ar[r]        & 1\\
1 \ar[r]  &  \ker \bar\psi_k \ar[r]                       & B_{k,n}(\surf{g}{})/ \Gamma_3(B_{k,n}(\surf{g}{}))  \ar[r]^{ \bar\psi_k} &  B_n(\surf{g}{})/ \Gamma_3(B_n(\surf{g}{}))  \ar[r]   & 1}}
\end{xy}
\end{equation}
where, by abuse of notation, $r_{k,n}$ and $r_n$ denote the canonical projections and $\bar\psi_k$ is the map such that 
$\bar\psi_k \circ r_{k,n}=r_n \circ \psi_k$. 
Following  \relem{presgamma}, we denote  
the generators of $\MB{k}{\surf{g}{}}/\Gamma_3(\MB{k}{\surf{g}{}})$ by $\si, \tsi, \zeta, a_1, b_1, \ldots, a_g, b_g, \tia_1, \tib_1, \ldots, \tia_g, \tib_g$.

\begin{prop} \label{fondhk}
The group $\ker \bar\psi_k$ admits the following presentation:\\

\noindent\textbf{Generating set:} $\sigma, \zeta, a_1, $ $ b_1, \ldots, a_g, b_g$;\\
\textbf{Relations:} 
\begin{enumerate}
\item $ [\sigma,a_i] = [\sigma,b_i] = 1$;
\item $[a_i,a_j] = [a_i,b_j] = 1$ for  $i\neq j$ and  $[a_j,b_j] = \sigma^2$;
\item $ [\zeta,a_i] = [\zeta,b_i] = [\zeta,\sigma] = 1$.
\end{enumerate}
\end{prop}

\begin{proof}
Let us denote  the subgroup of $B_{k,n}(\surf{g}{})/ \Gamma_3(B_{k,n}(\surf{g}{}))$ 
generated by  $\sigma, \zeta, a_1, $ $ b_1, \ldots,$ $a_g, b_g$ by $K_g$.
From diagram (\ref{diagramexactsequence}),  one deduces that  the group $K_g$ 
belongs to $\ker \bar\psi_k$.
By Lemma \ref{lem:presgamma},
 the group $K_g$ is normal in $B_{k,n}(\surf{g}{})/ \Gamma_3(B_{k,n}(\surf{g}{}))$.  A straightforward calculation shows that the quotient 
$B_{k,n}(\surf{g}{})/ \Gamma_3(B_{k,n}(\surf{g}{}))$ by $K_g$ is isomorphic to  $B_n(\surf{g}{})/ \Gamma_3(B_n(\surf{g}{}))$,
and therefore $K_g$ coincides with $\ker \bar\psi_k$.
To prove the proposition, it is therefore sufficient to check  that the given set of relations is a complete set of relations for $K_g$.
Let $D_g$ be the (abstract) group with the group presentation given in the statement. Let $j: D_g \to K_g$ be the (surjective) homomorphism sending every generator of $D_g$ to the corresponding generator of $K_g$, and let $\iota:K_g \to B_{k,n}(\surf{g}{})/ \Gamma_3(B_{k,n}(\surf{g}{}))$  be the (injective) homomorphism given in diagram
(\ref{diagramexactsequence}).
We claim that the composition $k= \iota \circ j$ is injective, and therefore $D_g$ coincides with $K_g$. In fact, 
it follows from the group presentation of $D_g$ that 
any element  $w \in  D_g$ can be written  (in a unique way) as $w=\sigma^p \zeta^q \prod_{j=1}^{g} a_j^{m_j}  \prod_{j=1}^{g}  b_j^{n_j} $.  We call this decomposition of $w$ its normal form in $D_g$. On the other hand, any element  $w' \in  B_{k,n}(\surf{g}{})/ \Gamma_3(B_{k,n}(\surf{g}{})) $ may be written  uniquely as
$w'=\sigma^p \zeta^q \tsi^r \prod_{j=1}^{g}  (a_j^{m_j} \tia_j^{M_j}) \prod_{j=1}^{g} (b_j^{n_j}\tib_j^{N_j}) $. We call this decomposition of $w'$ its normal form in $B_{k,n}(\surf{g}{})/ \Gamma_3(B_{k,n}(\surf{g}{}))$.    Let  $w\in D_g$ be written in its normal form. From the definition of $k: D_g \to B_{k,n}(\surf{g}{})/ \Gamma_3(B_{k,n}(\surf{g}{}))$,
$k(w)$ coincides with its normal form in $B_{k,n}(\surf{g}{})/ \Gamma_3(B_{k,n}(\surf{g}{}))$. Therefore $k(w)=1$ implies that $w=1$.
\end{proof}

\begin{rem}\label{remfond}
Let $k\ge 3$. Then the group~$\ker \bar\psi_k$ is the group $G_\Sigma$ introduced in~\cite[Section~3]{HK}. The homomorphism~$\qk: B_k(\surf{g}{, n})\to \ker \bar\psi_k$ is the homomorphism $\varPhi_\Sigma: B_k(\surf{g}{, n})\to G_\Sigma$ also defined in~\cite[Section~3]{HK}.
\end{rem}

The homomorphism $\varPhi_\Sigma: B_k(\surf{g}{, n})\to G_\Sigma$ and the group $G_\Sigma$ defined in Section~3 of \cite{HK} were constructed in a technical manner in order to obtain
(non-linear) representations of surface braid groups subject to certain homological constraints (Definition 2.2 of  \cite{HK}).   Proposition \ref{fondhk} and Remark \ref{remfond} 
show that such technical constructions coincide with  objects arising in the lower central series. On the other hand, reinterpreting Lemma 3.1 and Theorem 4.3 of \cite{HK},
we conclude that  the unique  lifting extension of $p_k:  B_{k}(\D_n)$ to $B_k(\surf{g}{, n})$ arises from the lower central series and the  homomorphism  $\qk: B_k(\surf{g}{, n})\to \ker \bar\psi_k$, as we shall now explain in the following two propositions.

\begin{prop} \label{fondpropex}
Let $k\ge 3$, and let $\bar\psi_k: B_k(\surf{g}{, n})  \to \ker \bar\psi_k$ be the homomorphism defined above.
The homomorphism  $\bar\psi_k$ lifts to an action on the level of regular coverings.
\end{prop}

\begin{proof}
We recall that  there is an induced  action of $B_n(\surf{g}{})$ on the homology of the covering space of $\FF_k(\surf{g}{, n})/S_k$ $\beta \in B_n(\surf{g}{})$ if there exists a homomorphism $\bar\beta_*: G_k(\surf{g}{}) \to G_k(\surf{g}{})$
  such that $\bar\beta_* \circ \bar\psi_k= \bar\psi_k  \circ \beta_* $.  As we remarked previously, we may replace 
 the homomorphism  $\qk: B_k(\surf{g}{, n})\to \ker \bar\psi_k$ and the group
$\ker \bar\psi_k$  respectively by the homomorphism $\varPhi_\Sigma: B_k(\surf{g}{, n})\to G_\Sigma$ and the group $G_\Sigma$ given in Section 3 of \cite{HK}. The statement follows from Lemma 3.1 of  \cite{HK}. 
\end{proof}

\begin{prop} \label{fondprop}
Let $k\ge 3$.
\begin{enumerate}[(i)]
\item The homomorphism  $\bar\psi_k: B_k(\surf{g}{, n})  \to \ker \bar\psi_k$ is a lifting extension of $p_k:  B_{k}(\D_n)
  \to G_k$;
  \item let $\qk: B_k(\surf{g}{, n})  \to G_k(\surf{g}{})$ be a lifting extension of $p_k:  B_{k}(\D_n)
  \to G_k$.  There is a canonical   isomorphism $\iota:G_k(\surf{g}{}) \to \ker \bar\psi_k$ such that  $\iota \circ \qk$ sends any generator
  of $B_k(\surf{g}{, n})$ to the corresponding coset of $\MB{k}{\surf{g}{}}/\Gamma_3(\MB{k}{\surf{g}{}})$.
\end{enumerate} 
\end{prop}

\begin{proof}
\begin{enumerate}[(i)]
\item In order to prove that $\bar\psi_k: B_k(\surf{g}{, n})  \to \ker \bar\psi_k$ is a lifting extension of $p_k:  B_{k}(\D_n)
  \to G_k$, it suffices to show that $\bar\psi_k(B_{k}(\D_n))$ is isomorphic to $\Z^2$  in $\ker \bar\psi_k$.
  This is immediate because $\bar\psi_k(\sigma_j)=\sigma$ for any $j=1,\ldots,k-1$, and
  $\bar\psi_k(\zeta_l)=\zeta$ for any $l=1,\ldots,n$. The claim then follows from Remark \ref{rem_embed}

\item
Let $\qk: B_k(\surf{g}{, n})  \to G_k(\surf{g}{})$ be a lifting extension of $p_k:  B_{k}(\D_n)
  \to G_k$.  Therefore $\qk (\sigma_j)=q$ for all $j=1,\ldots,k-1$, and
  $\qk (\zeta_l)=t$ for all $l=1,\ldots,n$, where $q,t \in G_k(\surf{g}{})$ generate a central subgroup isomorphic to 
  $\Z^2$.  One can check easily that these conditions imply that:
  \begin{itemize}
  \item  the images of $a_1, b_1, \ldots, a_g, b_g$
  commute with $q$ and $t$;
  \item $[\qk(a_i),\qk(b_j)]=1$ for $i \not=j$;
  \item $[\qk(a_i),\qk(b_i)]=q^2$.
  \end{itemize}
  
We can therefore define a  surjection $\bar{p}_k:\ker \bar\psi_k \to G_k(\surf{g}{})$ such that
  $\bar{p}_k(\sigma)=q, \bar{p}_k(\zeta)=t$ and $\bar{p}_k(a_i)=\qk(a_i)$, $\bar{p}_k(b_i)=\qk(b_i)$.
  
We will prove that $\bar{p}_k$ is actually an isomorphism. Let  $w \in \ker \bar\psi_k$,and consider $1\leq r\leq g$. By Proposition \ref{fondhk}, we observe that for every generator~$x$ of the presentation, we have $a_r x = x a_r$ except for $x = b_r$, in which case we have $ a_r b_r=\sigma^2 b_r a_r$. Therefore, if the normal form of $w$ is $\sigma^c   \zeta^d \prod_{j=1}^{g} b_j^{n_j} a_j^{m_j}$, then $a_rw = \sigma^{2n_r}wa_r$ and $[w,a_r]= \sigma^{-2n_r}$. Similarly, $[w,b_r]= \sigma^{2m_r}$. Now if $\bar{p}_k(w)=1$ then $\bar{p}_k([w,a_r])=  \bar{p}_k([w,b_r])= 1$. We obtain $\bar{p}_k(\sigma^{-2n_r})=q^{-2n_r}=1$ and $\bar{p}_k(\sigma^{-2m_r})=q^{-2m_r} = 1$, and thus $m_r = n_r = 0$. Since this is so for all $r$, we have $w = \sigma^c\zeta^d$ and $\bar{p}_k(w) = q^ct^d = 1$. Since $q,t$ generate a subgroup that is isomorphic to $\mathbb{Z}^2$, we get $c = d = 0$ and $w = 1$. Hence $\bar{p}_k$ is injective.
\end{enumerate}
\end{proof}

The cases $k=1,2$ are still open, mainly because we do not have a finite group presentation for 
$\ker \bar\psi_k$ and  $B_{k,n}(\surf{g}{})/ \Gamma_3(B_{k,n}(\surf{g}{}))$  in these cases. As in Proposition \ref{fondprop}, there is a natural surjection from $\ker \bar\psi_k$ onto the group $G_\Sigma$ considered in Section 3 of \cite{HK}, but it is not an isomorphism.

\begin{rem}
Notice that in the proof of Proposition \ref{fondprop}, we proved a slightly stronger result, i.e.\ that if
$H_k$ is a group for which there exists a surjection $h_k: B_k(\surf{g}{, n})  \to H_k$ satisfying $h_k (\sigma_j)=q$ for all $j=1,\ldots,k-1$ and
  $h_k (\zeta_l)=t$ for all $l=1,\ldots,n$ , where the subgroup $\langle q,t\rangle$ of $H_{k}$ is torsion free, then the group $H_k$ coincides with $\ker \bar\psi_k$.
\end{rem}

Using Lemma \ref{lem:presgamma}  and Theorem \ref{th:presbng}, one may easily adapt the proof of Proposition \ref{fondprop}  in order to obtain the following result:

\begin{prop} \label{fondpropcons}
Let $k,n\geq 3$. 

\begin{enumerate}[(i)]
\item Let $H$ be a group, and let $\lambda_{\surf{g}{}}$:
$ B_{n}(\surf{g}{}) \to H$ be a surjective homomorphism such that $\lambda_{\surf{g}{}}(B_n)$ is isomorphic to~$\mathbb{Z}$. Then there is an isomorphism $\iota: H \to B_{n}(\surf{g}{})/ \Gamma_3(B_{n}(\surf{g}{}))$ for which  $\iota\circ \lambda_{\surf{g}{}}$ sends every generator
  of $ B_{n}(\surf{g}{})$ to the corresponding class of $ B_{n}(\surf{g}{})/ \Gamma_3(B_{n}(\surf{g}{}))$.
  
\item Let $H$ be a group. Assume that~$h: B_{k,n}(\surf{g}{}) \to H$ is a surjective homomorphism such that the image of the subgroup generated by $\sigma_1,\ldots,\sigma_{k-1},\tsi_1,\cdots,$ $\tsi_{n-1}$ is isomorphic to~$\mathbb{Z}^2$. 

Then there is an isomorphism $\iota:H \to  \displaystyle \frac{B_{k,n}(\surf{g}{})}{ \Gamma_3(B_{k,n}(\surf{g}{}))}$ for which  $\iota \circ h$ sends any generator
  of $ B_{k,n}(\surf{g}{})$ to the corresponding coset of $ \displaystyle \frac{ B_{k,n}(\surf{g}{})}{ \Gamma_3(B_{k,n}(\surf{g}{})}$.
\end{enumerate} 
\end{prop}

The above proposition is of additional interest in the study of the lower central series of surface braid groups, since it states 
that the only possible extension of the length function from $B_n$ to $B_n((\surf{g}{})$ is the canonical projection of
$B_{n}(\surf{g}{})$ onto $B_{n}(\surf{g}{})/ \Gamma_3(B_{n}(\surf{g}{}))$. It is interesting to remark that in the case of closed surfaces, we have the following result:

\begin{prop} \label{fondpropcons2}
Let $n \ge  3$ and let $\Sigma$ be an orientable surface of positive genus.
It is not possible to extend  the length function $\lambda: B_n \to \Z$ to $ B_{n}(\Sigma) $. In other words
there is no surjection $\lambda_{\Sigma}$
of $ B_{n}(\Sigma) $ onto a group $F$ such that the restriction of $\lambda_{\Sigma}$ to $B_n$ coincides with
$\lambda$. 
\end{prop}

\begin{proof}
Let $\lambda_{\Sigma}: B_{n}(\Sigma) \to F$ be such that $\lambda_{\Sigma}(\sigma_1)=\ldots
\lambda_{\Sigma}(\sigma_{n-1})$. Set $\lambda_{\Sigma}(\sigma_1)=\sigma$.
Using the group presentation of $ B_{n}(\Sigma)$  given in \cite{B}, we
see that $\sigma^{2(n+g-1)}=1$. For further details, we refer the reader to the calculation of the group presentation
of $B_{n}(\Sigma)/ \Gamma_3(B_{n}(\Sigma)$  given in the proof of Theorem 1 in \cite{BGG}.
\end{proof}


\section{Appendix on exact sequences}


Let us recall that the short exact sequences $(MBS)$ and $(PBS)$ also  exist if $\Sigma$ is a non-orientable surface.
If $n\geq 2$ and $k\geq 2$ then the splitting
of~(\ref{eq:sequence}) for compact surfaces without boundary (also possibly non-orientable) is an
open question. As the following example shows, the splitting of one of
the two short exact sequences~(\ref{eq:psequence})
and~(\ref{eq:sequence}) does not imply in general that the other
sequence splits.

\begin{exo}
\begin{enumerate}
\item Let $\Sigma=\mathbb{R}P^2$ and $n=k=2$. First note that the pure
braid sequence~(\ref{eq:psequence}) does not split by~\cite[Theorem
3]{GG3}. Secondly, using Van Buskirk's presentation of
$B_{n}(\mathbb{R}P^2)$~\cite{VB} in the case $n=4$, let
$a=\sigma_{3}^{-1}\sigma_{2}^{-1}\sigma_{1}^{-1} \rho_{1}$ (which is
of order $16$ by~\cite{GG3}) and let $\Delta_{4}=
\sigma_{1}\sigma_{2}\sigma_{3}  \sigma_{1}\sigma_{2} \sigma_{1}$
denote the `half-twist' braid (which is of order $4$ by~\cite{VB}).
Then~\cite[Proposition 15(a)]{GG7}  and standard
properties of dicyclic groups imply that the subgroup $H$ of
$B_{4}(\mathbb{R}P^2)$ generated by $a^{2}$ and $a\Delta_{4}$ is
isomorphic to the generalised quaternion group $\mathcal{Q}_{16}$ of
order $16$. Let
$\tau_{4}: B_{4}(\mathbb{R}P^2) \to S_{4}$ denote the usual permutation
homomorphism. A straightforward calculation shows that
$\tau_{4}(H)=\left\langle (1,3),(2,4)\right\rangle$, which is
isomorphic to $S_{2}\times S_{2}$. Taking $B_{2,2}(\mathbb{R}P^2)$ to
be $\tau_{4}^{-1}(\left\langle (1,3),(2,4) \right\rangle)$, the
restriction  to $H$ of the projection
$p: B_{2,2}(\mathbb{R}P^2) \to B_{2}(\mathbb{R}P^2)$ given
geometrically by forgetting the second and fourth strings is an
isomorphism. This follows since $\ker(p)$ is torsion free and
$B_{2}(\mathbb{R}P^2)\cong \mathcal{Q}_{16}$~\cite{VB}. In particular,
$p$ admits a section. So in this case, (\ref{eq:sequence}) splits,
but~(\ref{eq:psequence}) does not.

\item Let $\Sigma=\mathbb{S}^2$. For all $n\geq 3$ and all $k\in \N$,
(\ref{eq:psequence}) splits~\cite{Fa}. The question of the splitting
of~(\ref{eq:sequence}) is examined in~\cite{GG4}. For example, if
$n=3$ and $k\equiv 1 \bmod 3$, or if $n\geq 4$ and $k\not\equiv
\epsilon_{1}(n-1)(n-2)-\epsilon_{2}n(n-2) \bmod n(n-1)(n-2)$, where
$\epsilon_{1},\epsilon_{2}\in \brak{0,1}$, then~(\ref{eq:sequence})
does not split. So for these values of $k$ and $n$,
(\ref{eq:psequence}) splits, but~(\ref{eq:sequence}) does not.
\end{enumerate}
\end{exo}

Let us finish this appendix by pointing out another interesting feature when we pass from 
short exact sequences of classical braid groups to short exact sequences of surface braid groups.
The sequence $(PB)$ (in the case $k=1$) has the property that  the induced action of the quotient on the
Abelianisation of the kernel is trivial.  Following~\cite{FR},  we will call
such a splitting extension an \emph{almost-direct} product. The important remark for us in Theorem 3.1 of \cite{FR} is that such an exact sequences induces exact sequences on the level
of the lower central series  quotients (see also \cite{GP,Pap}).  Since $P_n$ is an iterated almost-direct product of free groups,
$P_n$ `inherits' various properties of $F_n$, and it is possible to use this structure to derive a presentation for the Lie algebra associated to the lower central series of $P_n$ and to construct a universal finite type invariant for braid groups~\cite{Pap}.
On the other hand, the fact that $P_n$ acts trivially on the Abelianisation of $F_n$ allows us to compose the Artin representation with the Magnus representation, thus yielding the Gassner representation (we refer to  \cite{Bar} for the details).

\begin{prop}
Let $\Sigma$ be an orientable surface different from $\mathbb{S}^2$ and $\mathbb{T}^2$.
The sequence:
\begin{equation*} 
\begin{xy}*!C\xybox{%
\xymatrix{
1 \ar[r] & P_{k}(\Sigma \setminus \sset{x_{1}}{x_{n}}) \ar[r] &
P_{k+n}(\Sigma) \ar[r]  & P_{n}(\Sigma) \ar[r] & 1 \\}}
\end{xy}
\end{equation*} defines an almost-direct product structure for $P_{k+n}(\Sigma)$ if and only if $n=1$.
\end{prop}

\begin{proof}
The case $n=1$ was proved in~\cite{BB2}.  If $n \ge 2$,
as we recalled in Section 3,  the sequence $(PBS)$  splits only if $\Sigma$ has  boundary
(Lemma~\ref{lem:fibrations}). However,  even when the sequence $(PBS)$ splits, the extension is never an almost-direct product. In fact,
a straightforward calculation on the level of group presentations (given for instance in \cite{B,BGG}) shows that the natural section (corresponding to adding a strand `at infinity') does not define a trivial action on the level of Abelianisation (see also relation (\emph{v}) of Proposition \ref{prop:presMB}).
The statement then follows from the following proposition, which shows that 
 the existence of an almost-direct product structure is
independent of the choice of section.
\end{proof}

\begin{prop}
Let $1\to K\to G\stackrel{p}{\to} Q\to 1$ be a split extension of
groups. Let $s,s'$ be sections for $p$, and suppose that the induced
action of $Q$ on $K$ via $s$ on the Abelianisation
$K^{\text{Ab}}=K/[K,K]$ is trivial. Then the same is true for the
section $s'$.
\end{prop}

\begin{proof}
Let $k\in K$ and $q\in Q$. By hypothesis, $s(q)\, k \,
(s(q))^{-1}\equiv k \bmod{[K,K]}$. Let $s'$ be another section for
$p$. Then $p\circ s'(q)=p\circ s(q)$, and so $s'(q) \, (s(q))^{-1}\in
\ker p$. Thus there exists $k'\in K$ such that $s'(q)=k'\, s(q)$, and
hence
\begin{equation*}
s'(q)\, k \, (s'(q))^{-1}\equiv k'\, s(q)\,k\, (s(q))^{-1} \,
k'^{-1}\equiv k'kk'^{-1}\equiv k \bmod{[K,K]}.
\end{equation*}
Thus the induced action of $Q$ on $K^{\text{Ab}}$ via $s'$ is also trivial.
\end{proof}

\vspace{10pt}

\noindent Paolo BELLINGERI,  Eddy GODELLE and  John GUASCHI,

\noindent Laboratoire de Math\'ematiques Nicolas Oresme, CNRS UMR 6139, Universit\'e de Caen BP 5186,  F-14032 Caen (France).

{\small

\noindent \textit{Email}:  paolo.bellingeri@unicaen.fr, eddy.godelle@unicaen.fr, john.guaschi@unicaen.fr}

\end{document}